\documentclass[a4paper,11pt]{article}

\usepackage{authblk}
\usepackage{hyperref}

\title{Estimation of the multifractional function and the stability index
 of linear multifractional stable processes\footnote{Declarations of interest:
none} }
\author{Thi To Nhu DANG}
\affil{University of Economics, The university of Danang\\ 71 Ngu Hanh Son Street, Danang, Vietnam. E-mail: nhudtt@due.edu.vn}

\usepackage[utf8x]{inputenc}
\usepackage{lmodern}
\usepackage{fancyhdr}
\usepackage[top=2cm, bottom=3cm, left=2cm, right=1.5cm]{geometry}
\usepackage{setspace}
\usepackage{textcomp}
\usepackage{graphicx}
\usepackage{subfig}
\usepackage[nottoc, notlof, notlot]{tocbibind}
\usepackage{url}
\usepackage{color}
\usepackage{bbm}

\usepackage{amsthm}
\usepackage{amsmath}
\usepackage{amssymb}
\usepackage{mathrsfs}
\usepackage{dsfont}

\usepackage{listings}				
{\theoremstyle{definition}
\newtheorem{defn}{Definition}[section]}

\newtheorem{thm}{Theorem}[section]
\newtheorem{remark}{Remark}[section]

\newtheorem{lem}{Lemma}[section]

\begin{document}

\maketitle
\begin{abstract}
 In this paper we are interested in multifractional stable processes where the self-similarity index $H$ is a function of time, in other words $H$ becomes time changing, and the stability index $\alpha$ is a constant.
Using $\beta$- negative power variations ($-1/2<\beta<0$), we propose estimators for the value of the  multifractional function $H$ at a fixed time $t_0$ and for $\alpha$ for two cases: multifractional Brownian motion ($\alpha=2$) and linear multifractional stable motion ($0<\alpha<2$). We get the consistency of our estimates for the underlying processes with the rate of convergence.
\end{abstract}
\tableofcontents
\section{Introduction}
Multifractional processes have been presented in order to overcome some limitations for some applications of the fractional Brownian motion because of the constancy in time of its self-similarity index $H$. In these processes, the path regularity can now vary with the time variable $t$. The well-known example is mutilfractional Brownian motion which was introduced by A. Benassi {\it{et al.}} in \cite{Benassi1997} and independently by R.F. Peltier and J. L\'evy V\'ehel in \cite{Peltier1995}, where the self-similarity index $H$ of fractional Brownian motion is replaced by a multifractional function $H(t)$, permitting the Hurst index to change in a prescribed manner. This flexible stochastic model permits local regularity and long range dependence to be separated to give sample paths that are both highly correlated and irregular. In the last twenty years, many multifractional processes have been introduced and investigated, see e.g., \cite{Ayache2007}, \cite{Bardet2013}, \cite{Benassi1998}, \cite{Coeurjolly2005}, \cite{Coeurjolly2006}, \cite{Embrechts2002}, \cite{Falconer2009}, \cite{Vehel2009}, \cite{Guevel2013}, \cite{Yimin2008}, \cite{Taqqu1994}, \cite{Stoev2002}, \cite{Stoev2004}, \cite{Stoev2006}.

 Therefore, the statistical estimation of the multifractional function $H$ at a value of variable time  $t$ for multifractional processes, 
has been interested by many authors since about two decades. In the statistical literature on this topic, the value of $H(.)$ at a fixed time $t_0$, is built via  \cite{Ayache2015}, \cite{Ayache2004}, \cite{Benassi2000},  \cite{Benassi1998}, \cite{Coeurjolly2005}, \cite{Coeurjolly2006}, \cite{Lacaux2004},  \cite{Guevel2013}. One can mention the work of R.F. Peltier and J. L\'evy V\'ehel (see \cite{Peltier1995}) for the estimation of the multifractional function of a multifractional Brownian motion, based on the average variation of the sampled process. In the case of Gaussian multifractional Brownian motion, strongly consistent estimators of $H(t_0)$ has been presented in \cite{Benassi1998}, using generalized quadratic variations of this process. For a more general Gaussian setting than that of the latter process, the increment ratio method is used to get the estimation of $H(t_0)$, see e.g, \cite{Bardet2013}.  Recently, the corresponding
estimation problem of the stability function and the multifractional 
function for a class of multistable processes was considered in the
discussion paper of R. Le Gu\'{e}vel, see \cite{Guevel2013}, based
on some conditions that involve the consistency of the estimators. For
linear multifractional stable motions, in \cite{Ayache2015}, the
authors presented strongly consistent estimators of the multifractional function $H(.)$ and the stability index $\alpha$ using wavelet
coefficients when $\alpha\in(1,2)$ and $H(.)$ is a H\"{o}lder function
smooth enough, with values in a compact subinterval $[\underline{H},
\overline{H}]$ of $(1/\alpha, 1)$. One can refer to \cite{Ayache2017}, in the setting of the symmetric $\alpha$-stable non-anticipative moving average linear multifractional stable motion, for an almost surely and $L^p(\Omega), p\in (0,4]$, consistent estimator of the multifractional function $H(.)$ when $\alpha\in (1,2)$.

 The aim of this work is to construct consistent estimators for the value of the multifractional function $H(.)$ at an arbitrary fixed time $t_0$ and for the stability index $\alpha$, using $\beta$-negative power variations ($-1/2<\beta<0$) for  multifractional Brownian motions ($\alpha=2$) and linear multifractional stable motions ($0<\alpha<2$). This framework has been introduced recently in a paper by T.T.N. Dang and J. Istas (see \cite{DanIst2017}) to estimate the Hurst and the stability indices of a $H$-self-similar stable process, in the context $H$ and $\alpha$ are constants, based on the fact that $\beta$-negative power variations have  expectations and covariances for $-1/2</\beta<0$. The authors showed that using these variations, one can obtain the estimate of $H$ without a priori knowledge on $\alpha$ and vice versa, the estimator of $\alpha$ can be ascertained without assumptions on $H$. 
 In this paper, to estimate the value of $H(.)$ at $t_0$, using this new framework requires no a priori knowledge on $\alpha$, but only a weak a priori condition on the supremum of function $H(t)$. We also get the consistent estimator for the stability index $\alpha$ for the underlying processes. Moreover, the rate of convergence of our estimates is given.
  
  The remainder part of this article is organized as follows: in the next
section, we present the setting and main results to
construct the estimators for $H(.)$ at a fixed time $t_0$  and for $\alpha$ for two cases: multifractional Brownian motion ($\alpha=2$) and linear multifractional stable motion ($0<\alpha<2$). In Section \ref{Proofs}, we gather all the proofs of the main results presented in Section \ref{Settings and main results}. These proofs make use of several lemmas which are introduced and proved in Subsection \ref{auxiliary results}.
  
\section{Settings and main results}\label{Settings and main results}
\begin{defn}\emph{Linear multifractional stable motion and multifractional Brownian motion.}\\
Let $0<\alpha\leq 2$ and $H: U\rightarrow (0,1)$ be an infinite differentiable function on a closed interval $U\subset \mathbb{R}$. Let
\begin{align}\label{eq.1}
X(t)=\int_\mathbb{R} (|t-s|^{H(t)-1/\alpha}-|s|^{H(t)-1/\alpha})M_\alpha (ds)
\end{align}
where $M_\alpha$ is a symmetric $\alpha$-stable random measure on $\mathbb{R}$ which control measure $ds$ is Lebesgue measure.\\
 When $0<\alpha<2$, the process $X(t)$ is called {\it{a linear multifractional stable motion}} (see, e.g., \cite{Vehel2009}). The Gaussian case ($\alpha=2$) is covered where $M(du)$ is the standard Gaussian measure $W(du)$ on $\mathbb{R}$, the process then is called {\it{a multifractional Brownian motion}} (see, e.g., \cite{Stoev2006}).
 \end{defn}
 Let $X(t)$ be the process defined by (\ref{eq.1}) and $t_0$ be a fixed point in $U$. We now construct estimators of $H(t_0)$ and $\alpha$. \\
Let $L\geq 1, K\geq 1$ be fixed integers, $a=(a_0,\ldots,a_K)$ be a finite sequence with exactly $L$ vanishing first moments, that is for all $q\in   \{0,\ldots, L\}$, one has
\begin{align}\label{eq.2}
\sum\limits_{k=0}^K k^qa_k&=0,
\sum\limits_{k=0}^K k^{L+1}a_k\neq 0
\end{align}
with convention $0^0=1$. For example, here we can choose $K=L+1$ and 
\begin{align}\label{eq.3}
a_k=(-1)^{L+1-k}\frac{(L+1)!}{k!(L+1-k)!}.
\end{align}
 The discrete variations $\triangle_{p,n}X$ with respect to the sequence $a$ 
  are defined by
 \begin{align}\label{eq.4}
\triangle_{p,n}X=\sum\limits_{k=0}^K a_k X\left(\frac{k+p}{n}\right).
\end{align}
Let $\gamma$ be fixed such that 
\begin{align}\label{eq.5}
0<\limsup\limits_{t\in U} H(t)<\gamma<1.
\end{align}
 Define a set $\nu_{\gamma, n}(t_0)$ and its cardinal by
\begin{align}
\nu_{\gamma,n} (t_0)&=\{k\in\mathbb{Z}: \forall p=0,\ldots,K, |\frac{k+p}{n}-t_0|\leq \frac{1}{n^\gamma}\},\label{eq.6}\\
\upsilon_{\gamma,n} (t_0)&=\# \nu_{\gamma,n} (t_0).\label{eq.7}
\end{align}
Here we can choose $n\in\mathbb{N}$ large enough such that 
$$\{\frac{k+p}{n}, k\in\nu_{\gamma,n} (t_0), p=0,\ldots,K\} \subset U .$$
 Note that $\upsilon_{\gamma,n} (t_0)= [2n^{1-\gamma}-K] $
  or  $[2n^{1-\gamma}-K]+1$ depending on the parity of $[2n^{1-\gamma}-K]$.\\
   Let $\beta \in (-1/2,0)$ be fixed and
    \begin{align}
  V_{t_0,n}(\beta)&=\frac{1}{\upsilon_{\gamma,n} (t_0)}\sum\limits_{k\in \nu_{\gamma,n} (t_0)} |\triangle_{k,n}X|^\beta \label{eq.8}\\
  W_{t_0,n}(\beta)&=n^{\beta H(t_0)} V_{t_0,n}(\beta).\label{eq.9}
  \end{align}
 Let $\widehat{H}_n(t_0)$ is defined by
 \begin{align}\label{eq.10}
 \widehat{H}_n(t_0)&=\frac{1}{\beta}\log_2 \frac{V_{t_0,n/2}(\beta)}{V_{t_0,n}(\beta)}.
 \end{align}
 We will prove later that $\widehat{H}_n(t_0)$ is a consistent estimator of $H(t_0)$ of the multifractional stable process defined by (\ref{eq.1}) at a fixed time $t_0$ where $t_0\in U$.\\
 We now present a consistent estimator of $\alpha$.\\
We define
first auxiliary functions $\psi_{u,v},h_{u,v},\varphi_{u,v}$ before
introducing the estimator of $ \alpha$, where $u>v>0$.\\
Let $\psi_{u,v}$: $\mathbb{R^{+}}\times\mathbb{R^{+}}\rightarrow
\mathbb{R}$ be the function defined by
\begin{equation}
\label{eq.11}
\psi_{u,v}(x,y)=-v\ln x+u\ln y+C(u,v),
\end{equation}
where
\begin{align*}
C(u,v)=
&\frac{u-v}{2}\ln\pi+u\ln\left(
\Gamma(1+\frac{v}{2})\right) +v \ln\left( \Gamma(\frac{1-u}{2})\right)
\\
& -v \ln\left( \Gamma(1+\frac{u}{2})\right)
-u\ln\left( \Gamma(
\frac{1-v}{2})\right) .
\end{align*}
Let $h_{u,v}: (0,+\infty)\rightarrow(-\infty,0)$ be the function
defined by
\begin{equation}
\label{eq.12}
h_{u,v}(x)=u\ln\left( \Gamma(1+\frac{v}{x})\right) -
v\ln\left( \Gamma(1+
\frac{u}{x})\right).
\end{equation}
%
Let $\varphi_{u,v}:\mathbb{R}\rightarrow[0,+\infty)$ be the function
defined by
\begin{equation}
\label{eq.13}
\varphi_{u,v}(x)=
\begin{cases}
0
&\mbox{ if $x\geq0$}
\\
h^{-1}_{u,v}(x)
& \mbox{ if $x<0$}
\end{cases}
\end{equation}
where $h_{u,v}$ is defined as in (\ref{eq.12}). One can see \cite{DanIst2017} for the elementary results on functions $\psi_{u,v}, h_{u,v}, \varphi_{u,v}$.
Let $\beta_{1},\beta_{2}$ be in $\mathbb{R}$ such that $-1/2<\beta
_{1}<\beta_{2}<0$. The estimator of $ \alpha$ is defined by
\begin{equation}
\label{eq.14}
\widehat{\alpha}_{n}=\varphi_{-\beta_{1},-\beta_{2}}\left( \psi_{-\beta
_{1},-\beta_{2}}(V_{n}(\beta_{1}),V_{n}(\beta_{2}))\right) ,
\end{equation}
where $\psi_{u,v}, \varphi_{u,v}$ are defined as in (\ref{eq.11}) and
(\ref{eq.13}), respectively. \\
The principle results in this paper are presented in Theorems \ref{thm.1}, \ref{thm.2} and \ref{thm.3} as follows. Theorem \ref{thm.1} covers the case of multifractional Brownian motion ($\alpha=2$) and Theorem \ref{thm.2} covers the case of linear multifractional stable motion ($0<\alpha<2$) for a consistent estimator of the value of the multifractional function $H(.)$ at a fixed time $t_0$. Theorem \ref{thm.3} is devoted to a consistent estimator of the stability index $\alpha$ for those two cases. 
\subsection{Estimation of the multifractional function  \texorpdfstring{$H$}{}}\label{Estimation of the self-similarity functional parameter H}
We are in position to construct a consistent estimator, with the rate of convergence, for value of the multifractional function $H(.)$ at a fixed time $t_0$ for linear multifractional stable motion ($0<\alpha<2$) and multifractional Brownian motion ($\alpha=2$). 
We first give some definitions as follows.\\
For $n\in\mathbb{N}, k\in\mathbb{Z}, s\in U$, let  
\begin{align}
f(k,n,s)&= \sum\limits_{p=0}^K a_p|\frac{k+p}{n}-s|^{H(\frac{k+p}{n})-1/\alpha}
,\label{eq.21}\\
 g(k,n,s)&=\sum\limits_{p=0}^K a_p\left|\frac{k+p}{n}-s\right|^{H(t_0)-1/\alpha} \label{eq.22}.
\end{align}
Let $\beta \in (-1/2,0), t_0\in U$ be fixed, set
\begin{align}
M_{t_0}&=\left(\int\limits_{\mathbb{R}}\left| \sum\limits_{p=0}^K a_p (|p-s|^{H(t_0)-1/\alpha}\right|^\alpha ds\right)^{1/\alpha}.\label{eq.24}\\
M_{t_0,\beta}&=\frac{M_{t_0}^\beta C_\beta}{\sqrt{2\pi}}\int\limits_{\mathbb{R}}\frac{e^{-|y|^\alpha}}{|y|^{1+\beta}} dy, \label{eq.34}
\end{align}
where $C_\beta$ is defined by
\begin{align}\label{eq.35}
C_\beta=\frac{2^{\beta+1}\Gamma(\frac{\beta+1}{2})}{\Gamma(-\frac{\beta}{2})}.
\end{align}
\subsubsection{Multifractional Brownian motion  (\texorpdfstring {$\alpha=2$}{})}
We consider the case $\alpha=2$ of the multifractional stable process defined by (\ref{eq.1}): multifractional Brownian motion.
 For $t_0\in U$ fixed, we get a consistent estimator for $H(t_0)$ with rate of convergence $d_n$ defined by 
\begin{align}\label{eq.15}
d_n&=\begin{cases}
     n^{H(t_0)-\gamma}& \text{ if    } H(t_0)< \gamma\leq\frac{1+2H(t_0)}{3} \\
         n^{\frac{\gamma-1}{2}}& \text{ if    } \gamma>\frac{1+2H(t_0)}{3}. \\
      \end{cases}
\end{align}
The result is given via the following theorem.
\begin{thm}\label{thm.1}
Let $X$ be a multifractional Brownian motion defined by (\ref{eq.1}) with $\alpha=2$ and $M_\alpha(ds)$ is the standard Gaussian measure on $\mathbb{R}$.
Then we have
\begin{enumerate}
\item 
\begin{align}\label{eq.16}
\lim\limits_{n\to +\infty} W_{t_0,n}(\beta)&\stackrel{(\mathbb{P})}{=} M_{t_0,\beta}, W_{t_0,n}(\beta)-M_{t_0,\beta}=O_\mathbb{P}(d_n)
\end{align}
where $W_{t_0,n},M_{t_0,\beta}, d_n$ are defined by (\ref{eq.9}), (\ref{eq.34}) and (\ref{eq.15}), respectively.
\item 
\begin{align}\label{eq.17}
\lim\limits_{n\to+\infty}\widehat{H}_n(t_0)=H(t_0),\widehat{H}_n(t_0)-H(t_0)&=O_\mathbb{P}(d_n)
\end{align}
where $\widehat{H}_n(t_0)$ is defined by (\ref{eq.10}) and $O_{\mathbb{P}}$ is defined by:\\
$\bullet X_n=O_{\mathbb {P}}(1)$ iff for all $\epsilon>0$, 
there exists $M>0$ such that $\sup\limits_{n}\mathbb{P}(|X_n|>M)<\epsilon$,\\
$\bullet Y_n=O_{\mathbb {P}}(a_n)$ means $Y_n=a_nX_n$ with $X_n=O_{\mathbb {P}}(1)$.
\end{enumerate} 
\end{thm}
\begin{proof}
See Subsection \ref{Proof of Theorem 1}
\end{proof}
\subsubsection{Linear multifractional stable motion (\texorpdfstring {$0<\alpha<2$}{})}
We now investigate the case $0<\alpha<2$ of the linear multifractional stable motion defined by (\ref{eq.1}). Let $t_0$ be a fixed point in $U$. Let 
\begin{align}\label{eq.18}
d_n&=\begin{cases}
     n^{\frac{\alpha(H(t_0)-\gamma)}{4}}& \text{ if    } H(t_0)<L+1-\frac{2}{\alpha} \text{ and } H(t_0)< \gamma\leq\frac{2+\alpha H(t_0)}{2+\alpha}, \\
   n^{\frac{\gamma-1}{2}}& \text{ if    } H(t_0)<L+1-\frac{2}{\alpha} \text{ and } \gamma>\frac{2+\alpha H(t_0)}{2+\alpha}, \\   
     n^{\frac{\alpha(1-\gamma)(H(t_0)-(L+1))}{4}}& \text{ if    } H(t_0)>L+1-\frac{2}{\alpha} \text{ and } \gamma\geq\frac{L+1}{L+2-H(t_0)}, \\
     n^{\frac{\alpha(H(t_0)-\gamma)}{4}}& \text{ if    } H(t_0)>L+1-\frac{2}{\alpha} \text{ and } H(t_0)<\gamma<\frac{L+1}{L+2-H(t_0)}, \\      
  n^{\frac{\alpha(H(t_0)-\gamma)}{4}}& \text{ if    } H(t_0)=L+1-\frac{2}{\alpha} \text{ and } H(t_0)<\gamma<\frac{(L+1)\alpha}{2+\alpha}, \\
  n^{\frac{\gamma-1}{2}}\sqrt{\ln(n)}           & \text{ if    } H(t_0)=L+1-\frac{2}{\alpha} \text{ and } \gamma\geq\frac{(L+1)\alpha}{2+\alpha}.
      \end{cases}
\end{align}
We obtain a consistent estimator of $H(t_0)$ via the following theorem.
\begin{thm}\label{thm.2}
Let $X$ be a linear multifractional stable motion defined by (\ref{eq.1}) with $0<\alpha<2$.
For $t_0\in U$ fixed, we have
\begin{enumerate}
\item \begin{align}\label{eq.19}
\lim\limits_{n\to +\infty} W_{t_0,n}(\beta)&\stackrel{(\mathbb{P})}{=} M_{t_0,\beta}, W_{t_0,n}(\beta)-M_{t_0,\beta}=O_\mathbb{P}(d_n)
\end{align}
where $W_{t_0,n}, M_{t_0},\beta, d_n$ are defined by (\ref{eq.9}), (\ref{eq.34}) and (\ref{eq.18}), respectively.
\item 
\begin{align}\label{eq.20}
\lim\limits_{n\to+\infty}\widehat{H}_n(t_0)=H(t_0),\widehat{H}_n(t_0)-H(t_0)&=O_\mathbb{P}(d_n)
\end{align}
where $\widehat{H}_n(t_0)$ is defined by (\ref{eq.10}).
\end{enumerate}
\end{thm}
\begin{proof}
See Subsection \ref{Proof of Theorem 2}.
\end{proof}
\subsection{Estimation of the stable index \texorpdfstring {$\alpha$}{}}\label{Estimation of the stable index}
For the multifractional stable process defined by (\ref{eq.1}), we consider the stable index $\alpha$ as a constant and $\alpha \in (0,2]$. Now we will present a consistent estimator for $\alpha$ in this case. Recall that $\widehat{\alpha}_n$ is defined by (\ref{eq.14}):
\begin{equation*}
\widehat{\alpha}_n=\varphi_{-\beta_1,-\beta_2}\left(\psi_{-\beta_1,-\beta_2}(V_{t_0,n}(\beta_1),V_{t_0,n}(\beta_2))\right).
\end{equation*}
The results on estimating $\alpha$ is presented via the following theorem.
\begin{thm}\label{thm.3}
Let $X$ be a multifractional stable process defined by (\ref{eq.1}). Let $t_0\in U$ be fixed, then $\widehat{\alpha}_n$ is a consistent estimator of $\alpha$, moreover $\widehat{\alpha}_n-\alpha=O_\mathbb{P}(d_n)$, where  $\widehat{\alpha}_n$ is defined by (\ref{eq.14}), $d_n$ is defined by (\ref{eq.15}) for the case of multifractional Brownian motion ($\alpha=2$) and $d_n$ is defined by (\ref{eq.18}) for the case of linear multifractional stable motion ($ 0<\alpha<2$).
\end{thm}
\begin{proof}
See Subsection \ref{Proof of Theorem 3}.
\end{proof}
\section{Proofs}\label{Proofs}
This Section is devoted to the proofs of theorems presented in Section \ref{Settings and main results}. 
\subsection{Auxiliary results}\label{auxiliary results}
We present here some results related to discrete variations of linear multifractional stable motion and multifractional Brownian motion. These results will be used to prove main results.
\begin{lem}\label{lem.1.18}
Let $X$ be a multifractional stable process defined by (\ref{eq.1}). For $0<\alpha\leq 2$ and $k\in\nu_{\gamma,n}(t_0)$, let 
\begin{align}
\sigma_{k,n}(t_0)&=||\frac{\triangle_{k,n}X}{n^{-H(t_0)}}||_\alpha, \label{eq.23} 
\end{align}
with the notation $||f||_{\alpha}=(\int\limits_{S}|f(s)|^{\alpha}\mu (ds))^{1/\alpha}$, where $f\in L^{\alpha}(S,\mu)$.
Then 
\begin{align}\label{eq.25}
|\sigma_{k,n}(t_0)-M_{t_0}|=O(n^{\alpha(H(t_0)-\gamma)\wedge (H(t_0)-\gamma)}),
\end{align}
 where $\gamma, M_{t_0}$ are defined by (\ref{eq.5}) and (\ref{eq.24}), respectively. 
\end{lem}
\begin{remark}\label{remark.1.3}
From Lemma \ref{lem.1.18} and since $H(t_0)<\gamma$, it follows that 
$$
\lim\limits_{n\to +\infty} \sigma_{k,n}(t_0)=M_{t_0}.
$$
Therefore, there exist $n_0\in \mathbb{N}$ and constants $0<M_1<M_{t_0}<M_2$ such that for all $n\geq n_0$, then $M_1<\sigma_{k,n}(t_0)<M_2$.  
\end{remark}
\begin{proof}
We have
\begin{align*}
\sigma_{k,n}(t_0)=\left(\int\limits_{\mathbb{R}}\left| \sum\limits_{p=0}^K a_p n^{H(t_0)}(|\frac{k+p}{n}-s|^{H(\frac{k+p}{n})-1/\alpha}-|s|^{H(\frac{k+p}{n})-1/\alpha} )\right|^\alpha ds\right)^{1/\alpha}.
\end{align*}
By changing variable $s'=\frac{s+k}{n}$ and using the fact that $\sum\limits_{p=0}^K a_p=0$, one can write
\begin{align*}
M_{t_0}^\alpha&=\int\limits_{\mathbb{R}}\left| \sum\limits_{p=0}^K a_p n^{H(t_0)}\left(|\frac{k+p}{n}-s'|^{H(t_0)-1/\alpha}-|s'|^{H(t_0)-1/\alpha} \right)\right|^\alpha ds' \\
&=\int\limits_{\mathbb{R}}\left| \sum\limits_{p=0}^K a_p n^{H(t_0)}(|\frac{k+p}{n}-s|^{H(t_0)-1/\alpha}-|s|^{H(t_0)-1/\alpha} )\right|^\alpha ds. 
\end{align*}
Then following Lemma 4.7.2 in \cite{Taqqu1994}, one has
\begin{align}
|\sigma_{k,n}^\alpha (t_0)-M_{t_0}^\alpha|&=
\left| \int\limits_{\mathbb{R}}\left(|f(k,n,s)|^\alpha -|g_{k,n}^1(s)|^\alpha\right) ds  \right| \notag\\
&\leq  \int\limits_{\mathbb{R}}\left||f(k,n,s)|^\alpha -|g(k,n,s)|^\alpha\right|ds \notag\\
&\leq C_*\left( \int\limits_{\mathbb{R}}\left|
f(k,n,s)-g(k,n,s)\right|^\alpha ds \right)^{1\wedge 1/\alpha},\label{eq.26}
\end{align}
where 
\begin{align*}
C_*&= \begin{cases}
    1 & \text{ if } 0<\alpha\leq 1,\\
    2^{1/\alpha}\alpha (||f_1(k,n,s)||_\alpha^{\alpha-1}+||g_1(k,n,s)||_\alpha^{\alpha-1})& \text { if } 1<\alpha\leq 2,\\
  \end{cases}\\
f_1(k,n,s)&= \sum\limits_{p=0}^K a_p n^{H(t_0)}\left(|\frac{k+p}{n}-s|^{H(\frac{k+p}{n})-1/\alpha}-|s|^{H(\frac{k+p}{n})-1/\alpha} \right)   =n^{H(t_0)}f(k,n,s),\\
g_1(k,n,s)&= \sum\limits_{p=0}^K a_p n^{H(t_0)}\left(|\frac{k+p}{n}-s|^{H(t_0)-1/\alpha}-|s|^{H(t_0)-1/\alpha} \right)=n^{H(t_0)}g(k,n,s),
\end{align*}
$f(k,n,s), g(k,n,s)$ are defined by (\ref{eq.21}), (\ref{eq.22}), respectively.
We now will prove that $C_*$ can be bounded by a constant.\\
Using the change of variable $s=s_1/n$, then $s_2=s_1-k$, one gets
\begin{align}
\int_\mathbb{R}|g(k,n,s)|^\alpha ds&=n^{-\alpha H(t_0)}\int_\mathbb{R}|\sum\limits_{p=0}^Ka_p|k+p-s_1|^{H(t_0)-1/\alpha}|^\alpha ds_1 \notag\\
&=n^{-\alpha H(t_0)}\int_\mathbb{R}|\sum\limits_{p=0}^Ka_p|p-s_2|^{H(t_0)-1/\alpha}|^\alpha ds_2\notag \\
&=M_{t_0}^\alpha n^{-\alpha H(t_0)}. \label{eq.27}
\end{align}
It follows that 
\begin{align}\label{eq.28}
||g_1(k,n,s)||_\alpha=n^{H(t_0)}\left(\int\limits_{\mathbb{R}}|g(k,n,s)|^\alpha ds\right)^{1/\alpha}=M_{t_0}.
\end{align}
For $||f_1(k,n,s)||_\alpha $, applying Lemma 2.7.13 in \cite{Taqqu1994}, one obtains
\begin{align}
||f_1(k,n,s)||_\alpha^\alpha&=n^{\alpha H(t_0)}\int\limits_{\mathbb{R}}|f(k,n,s)|^\alpha ds \notag\\
&\leq 2^{0\wedge(\alpha-1)}n^{\alpha H(t_0)}\int\limits_{\mathbb{R}}\left(|f(k,n,s)-g(k,n,s)|^\alpha+|g(k,n,s)|^\alpha\right)ds \label{eq.29}
\end{align}
where $x\wedge y =\min \{x,y\}$.\\
Now we consider 
$\int_\mathbb{R}|f(k,n,s)-g(k,n,s)|^\alpha ds $. Applying again $K$ times Lemma 2.7.13 in \cite{Taqqu1994}, 
  for $0<\alpha\leq 2$, one gets
\begin{align}
&|f(k,n,s)-g(k,n,s)|^\alpha \notag\\
&\leq \sum\limits_{p=0}^K\left|a_p \left(|\frac{k+p}{n}-s|^{H(\frac{k+p}{n})-1/\alpha}-|s|^{H(\frac{k+p}{n})-1/\alpha}-|\frac{k+p}{n}-s|^{H(t_0)-1/\alpha}-|s|^{H(t_0)-1/\alpha}\right)\right|^\alpha \notag \\
&\times 2^{0\wedge (K(\alpha-1))} \notag\\
&=2^{0\wedge (K(\alpha-1))} \sum\limits_{p=0}^K|a_p|^\alpha|h(\frac{k+p}{n},\frac{k+p}{n},s)-h(\frac{k+p}{n},t_0,s)|^\alpha \label{eq.30}
\end{align}
where 
\begin{align*}
h(t,v,s)=|t-s|^{H(v)-1/\alpha}-|s|^{H(v)-1/\alpha}.
\end{align*}
Similarly to the proof of Theorem 7.4 in \cite{Vehel2009}, let $h_-,h_+$ be fixed such that $0<h_-<H(t)<h_+<1$ for all $t\in U$.
With $s\neq 0, s\neq \frac{k+p}{n}$, applying the mean value theorem, since $k\in \nu_{\gamma,n}(t_0)$, one obtains
\begin{align*}
&|h(\frac{k+p}{n},\frac{k+p}{n},s)-h(\frac{k+p}{n},t_0,s)|\\
&=|H(\frac{k+p}{n})-H(t_0)|\left||\frac{k+p}{n}-s|^{H(.)-1/\alpha}\ln|\frac{k+p}{n}-s|-|s|^{H(.)-1/\alpha}\ln|s|\right|\\
&\leq \frac{\sup\limits_{t\in U}H'(t)}{n^\gamma} \left||\frac{k+p}{n}-s|^{H(.)-1/\alpha}\ln|\frac{k+p}{n}-s|-|s|^{H(.)-1/\alpha}\ln|s|\right|\\
&\leq\frac{K_1(\frac{k+p}{n},s)}{n^\gamma},
\end{align*}
where $H(.)$ is on a line segment connecting $H(\frac{k+p}{n})$ and $H(t_0)$ and 
\begin{align*}
K_1(t,s)=\begin{cases}
 c_1 \max \{1, |t-s|^{h_--1/\alpha}+|s|^{h_--1/\alpha}\}& \text{ if } |s|\leq 1+ 2\max\limits_{t\in U}|t|,\\
c_2|s|^{h_+-1/\alpha-1}& \text{ if }  |s|> 1+ 2\max\limits_{t\in U}|t|\\
 \end{cases}
\end{align*}
where $c_1$ and $c_2$ are appropriate constants, see \cite{Vehel2009}. Then
$\int\limits_\mathbb{R}K_1(t,s)^\alpha ds<+\infty$ and uniformly bounded for $t\in U$. Combining with (\ref{eq.30}), it follows that there exists a constant $C>0$ such that
\begin{align}\label{eq.31}
\int_\mathbb{R}|f(k,n,s)-g(k,n,s)|^\alpha ds &\leq \frac{C}{n^{\alpha \gamma}}.
\end{align}
From (\ref{eq.27}), (\ref{eq.29}), (\ref{eq.31}) and since $H(t_0)<\gamma$, there exists a running constant $C>0$ such that 
\begin{align}\label{eq.32}
||f_1(k,n,s)||_\alpha&\leq C.
\end{align}
From (\ref{eq.28}) and (\ref{eq.32}), it follows that $C_*$ can be bounded by a constant.\\
We now consider $\int\limits_{\mathbb{R}}\left|
f_1(k,n,s)-g_1(k,n,s)\right|^\alpha ds
$. From (\ref{eq.31}), one gets
\begin{align*}
\int\limits_{\mathbb{R}}\left|
f_1(k,n,s)-g_1(k,n,s)\right|^\alpha ds&=n^{\alpha H(t_0)}\int\limits_{\mathbb{R}}\left|
f(k,n,s)-g(k,n,s)\right|^\alpha ds\\
&\leq Cn^{\alpha(H(t_0)-\gamma)}.
\end{align*}
It follows that
\begin{align}\label{eq.33}
|\sigma_{k,n}^\alpha (t_0)-M_{t_0}^\alpha|&=O \left(n^{\alpha(H(t_0)-\gamma)\wedge(H(t_0)-\gamma) }\right).
\end{align}
Thus, there exist $n_0\in\mathbb{N}$ and constants $0<M_1<M_{t_0}<M_2$ such that for all $n\geq n_0$, then  $0<M_1<\sigma_{k,n}(t_0)<M_2$.\\
For $0< \alpha\leq 2, \alpha\neq 1$, the mean value theorem gives
\begin{align*}
|\sigma_{k,n}^\alpha(t_0)-M_{t_0}^\alpha|&=\alpha|\sigma_{k,n}(t_0)-M_{t_0}|x_0^{\alpha-1}
\end{align*}
where $x_0\in(M_1,M_2)$. It follows that
\begin{align*}
|\sigma_{k,n}(t_0)-M_{t_0}|&\leq C n^{\alpha(H(t_0)-\gamma)\wedge(H(t_0)-\gamma) },
\end{align*}
 which means
$|\sigma_{k,n}(t_0)-M_{t_0}|=O\left(n^{\alpha(H(t_0)-\gamma)\wedge(H(t_0)-\gamma) }\right)$.
\end{proof} 
\begin{lem}\label{lem.1.19}
Let $X$ be a multifractional stable process defined by (\ref{eq.1}). For $0<\alpha\leq 2$ and $k\in\nu_{\gamma,n}(t_0)$, then
\begin{align*}
\left|\mathbb{E}\left|\frac{\triangle_{k,n}X}{n^{-H(t_0)}}\right|^\beta- M_{t_0,\beta}\right|&=O \left( n^{\alpha(H(t_0)-\gamma)\wedge(H(t_0)-\gamma) }\right)
\end{align*}
and
\begin{align*}
\left|\mathbb{E}W_{t_0,n}(\beta)- M_{t_0,\beta}\right|&=O \left( n^{\alpha(H(t_0)-\gamma)\wedge(H(t_0)-\gamma) }\right).
\end{align*}
\end{lem}
\begin{proof}
Since $\frac{\triangle_{k,n}X}{n^{-H(t_0)}}$ is a $S\alpha S$- stable random variable and $\sigma_{k,n}(t_0)=||\frac{\triangle_{k,n}X}{n^{-H(t_0)}}||_\alpha$, one gets
\begin{align*}
\mathbb{E}e^{iy\frac{\triangle_{k,n}X}{n^{-H(t_0)}}}&= e^{-|y|^{\alpha} \sigma_{k,n}(t_0)^\alpha}.
\end{align*}
Following Theorem 4.1 in \cite{DanIst2017} and using the change of variable $y_1=y\sigma_{k,n}$, one can write
\begin{align*}
\mathbb{E}\left|\frac{\triangle_{k,n}X}{n^{-H(t_0)}}\right|^\beta&=\frac{ C_\beta}{\sqrt{2\pi}}\int\limits_{\mathbb{R}}\frac{\mathbb{E}e^{iy\frac{\triangle_{k,n}X}{n^{-H(t_0)}}} }{|y|^{1+\beta}} dy=\frac{ C_\beta}{\sqrt{2\pi}}\int\limits_{\mathbb{R}}\frac{e^{-|y|^{\alpha} \sigma_{k,n}(t_0)^\alpha} }{|y|^{1+\beta}} dy\\
&= \frac{ \sigma_{k,n}^\beta(t_0) C_\beta}{\sqrt{2\pi}}\int\limits_{\mathbb{R}}\frac{e^{-|y_1|^{\alpha}} }{|y_1|^{1+\beta}} dy_1.
\end{align*}
It follows that
\begin{align*}
\left|\mathbb{E}\left|\frac{\triangle_{k,n}X}{n^{-H(t_0)}}\right|^\beta- M_{t_0,\beta}\right|&=  |\sigma_{k,n}^\beta(t_0)-M_{t_0}^\beta|\frac{C_\beta}{\sqrt{2\pi}}\int\limits_{\mathbb{R}}\frac{e^{-|y|^{\alpha}} }{|y|^{1+\beta}} dy \\
&=C|\sigma_{k,n}^\beta(t_0)-M^\beta|. 
\end{align*}
Applying the mean value theorem and Remark \ref{remark.1.3}, we get
\begin{align*}
|\sigma_{k,n}^\beta(t_0)-M_{t_0}^\beta|&=|\beta||\sigma_{k,n}(t_0)-M_{t_0}| \theta^{\beta-1}
\end{align*}
where $\theta \in (M_1, M_2)$.
 Combining with Lemma \ref{lem.1.18}, for $n\geq n_0$, it follows that
 \begin{align*}
 \left|\mathbb{E}\left|\frac{\triangle_{k,n}X}{n^{-H(t_0)}}\right|^\beta- M_{t_0,\beta }\right|& \leq C n^{\alpha(H(t_0)-\gamma)\wedge(H(t_0)-\gamma) }.
 \end{align*}
 Then 
$\left|\mathbb{E}\left|\frac{\triangle_{k,n}X}{n^{-H(t_0)}}\right|^\beta- M_{t_0,\beta}\right|=O \left( n^{\alpha(H(t_0)-\gamma)\wedge(H(t_0)-\gamma) }\right).$
One also gets
\begin{align*}
|\mathbb{E} W_{t_0,n}(\beta)-M_{t_0,\beta}|&= \left|\frac{1}{\upsilon_{\gamma,n}(t_0)}\sum\limits_{k\in\nu_{\gamma,n}(t_0)}\mathbb{E}\left(\left|\frac{\triangle_{k,n}X}{n^{-H(t_0)}}\right|^\beta-M_{t_0,\beta}\right)\right|\\
&\leq \frac{1}{\upsilon_{\gamma,n}(t_0)}\sum\limits_{k\in\nu_{\gamma,n}(t_0)}\left|\mathbb{E}\left|\frac{\triangle_{k,n}X}{n^{-H(t_0)}}\right|^\beta- M_{t_0,\beta}\right|.
\end{align*}
Since $\upsilon_{\gamma,n}=\# \nu_{\gamma,n}$, one obtains 
\begin{align*}
|\mathbb{E} W_{t_0,n}(\beta)-M_{t_0,\beta}|&=O \left( n^{\alpha(H(t_0)-\gamma)\wedge(H(t_0)-\gamma) }\right).
\end{align*}
\end{proof}
\begin{lem}\label{lem.1.20}
Let $X$ be a linear multifractional stable process defined by (\ref{eq.1}). Then there exists a constant $C>0$ such that for $k,k'\in\nu_{\gamma,n}(t_0)$, we have 
\begin{align*}
\left|cov\left(\left|\frac{\triangle_{k,n}X}{n^{-H(t_0)}}\right|^{\beta},\left|\frac{\triangle_{k',n}X}{n^{-H(t_0)}}\right|^{\beta}\right)\right|&\leq C.
\end{align*}
\end{lem}
\begin{proof}
Since $-1/2<\beta<0$, from Theorem 4.1 in \cite{DanIst2017} and Lemma \ref{lem.1.18}, for $n\geq n_0$, we can write 
\begin{align*}
\mathbb{E}\left|\frac{\triangle_{k,n}X}{n^{-H(t_0)}}\right|^{2\beta}&= \frac{ C_{2\beta}}{\sqrt{2\pi}}\int\limits_{\mathbb{R}}\frac{\mathbb{E}e^{iy\frac{\triangle_{k,n}X}{n^{-H(t_0)}}} }{|y|^{1+2\beta}} dy=\frac{ C_{2\beta}}{\sqrt{2\pi}}\int\limits_{\mathbb{R}}\frac{e^{-|y|^{\alpha} \sigma_{k,n}^\alpha} }{|y|^{1+2\beta}} dy\\
&= \frac{ \sigma_{k,n}^{2\beta} C_{2\beta}}{\sqrt{2\pi}}\int\limits_{\mathbb{R}}\frac{e^{-|y|^{\alpha}} }{|y|^{1+2\beta}} dy\leq \frac{ M_1^{2\beta} C_{2\beta}}{\sqrt{2\pi}}\int\limits_{\mathbb{R}}\frac{e^{-|y|^{\alpha}} }{|y|^{1+2\beta}} dy 
\end{align*}
where $C_{2\beta}$ is defined by (\ref{eq.35}) and $M_1$ is defined as in Remark \ref{remark.1.3}.
Applying Cauchy-Schwarz's inequality, one gets
\begin{align*}
\mathbb{E}\left|\frac{\triangle_{k,n}X}{n^{-H(t_0)}}\right|^{\beta}\left|\frac{\triangle_{k',n}X}{n^{-H(t_0)}}\right|^{\beta}&\leq \left(\mathbb{E}\left|\frac{\triangle_{k,n}X}{n^{-H(t_0)}}\right|^{2\beta}\mathbb{E}\left|\frac{\triangle_{k',n}X}{n^{-H(t_0)}}\right|^{2\beta}\right)^{1/2}\leq \frac{ M_1^{2\beta} C_{2\beta}}{\sqrt{2\pi}}\int\limits_{\mathbb{R}}\frac{e^{-|y|^{\alpha}} }{|y|^{1+2\beta}} dy .
\end{align*}
Moreover 
\begin{align*}
\mathbb{E}\left|\frac{\triangle_{k,n}X}{n^{-H(t_0)}}\right|^{\beta}&\leq \left(\mathbb{E}\left|\frac{\triangle_{k,n}X}{n^{-H(t_0)}}\right|^{2\beta}\right)^{1/2}\leq \left(\frac{ M_1^{2\beta} C_{2\beta}}{\sqrt{2\pi}}\int\limits_{\mathbb{R}}\frac{e^{-|y|^{\alpha}} }{|y|^{1+\beta}} dy \right)^{1/2}.
\end{align*}
Then we deduce
\begin{align*}
\left|cov\left(\left|\frac{\triangle_{k,n}X}{n^{-H(t_0)}}\right|^{\beta},\left|\frac{\triangle_{k',n}X}{n^{-H(t_0)}}\right|^{\beta}\right)\right|&\leq \mathbb{E}\left|\frac{\triangle_{k,n}X}{n^{-H(t_0)}}\right|^{\beta}\left|\frac{\triangle_{k',n}X}{n^{-H(t_0)}}\right|^{\beta}+\mathbb{E}\left|\frac{\triangle_{k,n}X}{n^{-H(t_0)}}\right|^{\beta}\mathbb{E}\left|\frac{\triangle_{k',n}X}{n^{-H(t_0)}}\right|^{\beta} \\
&\leq \frac{ 2M_1^{2\beta} C_{2\beta}}{\sqrt{2\pi}}\int\limits_{\mathbb{R}}\frac{e^{-|y|^{\alpha}} }{|y|^{1+\beta}} dy=C. 
\end{align*}
\end{proof}
\subsection{Proof of Theorem \ref{thm.1}}\label{Proof of Theorem 1}
To prove Theorem \ref{thm.1}, we first present the following lemma.
\begin{lem}\label{lem.1.22}
Let $X$ be a multifractional Brownian motion defined by (\ref{eq.1}) with $\alpha=2$,  $M_\alpha(ds)$ be the standard Gaussian measure on $\mathbb{R}$ and $t_0\in U$. Then there exist $n_1,k_0\in\mathbb{N}$, a constant $C>0$ such that for all $n\geq n_1, k,k'\in\nu_\gamma(t_0)$, and $|k-k'|>k_0$, we have 
\begin{align}\label{eq.36}
cov(|\triangle_{k,n}X|^\beta,|\triangle_{k',n}X|^\beta)\leq C n^{-2\beta H(t_0)}\left(n^{4(H(t_0)-\gamma)}+n^{2(H(t_0)-\gamma)}+|k-k'|^{2H(t_0)-2(L+1)}\right).
\end{align}
\end{lem}
\begin{proof}
Let
\begin{align*}
I&=\int\limits_\mathbb{R}\left|f(k,n,s)f(k',n,s)\right|ds
\end{align*}
where $f(k,n,s)$ is defined by (\ref{eq.21}).
 One has
 \begin{align*}
 |f(k,n,s)|&\leq |f(k,n,s)-g(k,n,s)|+|g(k,n,s)|,\\
 |f(k',n,s)|&\leq |f(k',n,s)-g(k',n,s)|+|g(k',n,s)|
 \end{align*}
 where $g(k,n,s)$ is defined by (\ref{eq.22}).
 It follows that
\begin{align}\label{eq.37}
I&\leq \int\limits_\mathbb{R}\left(|f(k,n,s)-g(k,n,s)|+|g(k,n,s)|\right)\left(|f(k',n,s)-g(k',n,s)|+|g(k',n,s)|\right)ds \notag\\
&:=I_1+I_2+I_3+I_4,
\end{align}
where
\begin{align*}
I_1&=\int\limits_\mathbb{R}|f(k,n,s)-g(k,n,s)||f(k',n,s)-g(k',n,s)|ds\\
I_2&=\int\limits_\mathbb{R}|f(k,n,s)-g(k,n,s)||g(k',n,s)|ds\\
I_3&=\int\limits_\mathbb{R}|f(k',n,s)-g(k',n,s)||f(k,n,s)|ds\\
I_4&=\int\limits_\mathbb{R}|g(k,n,s)||g(k',n,s)|ds.
\end{align*}
We will find the bound for each term $I_1, I_2,I_3, I_4$.\\ 
For $I_4$, one has
\begin{align*}
I_4&=\int\limits_\mathbb{R}|g(k,n,s)||g(k',n,s)|ds\\
&=\int_\mathbb{R}\left| \sum\limits_{p=0}^Ka_p|\frac{k+p}{n}-s|^{H(t_0)-1/2} \right|\left| \sum\limits_{p'=0}^Ka_{p'}|\frac{k'+p'}{n}-s|^{H(t_0)-1/2} \right| ds.
\end{align*}
Making the change of variable $s=s_1/n$, one obtains
\begin{align*}
I_4&=n^{- 2H(t_0)}\int\limits_\mathbb{R}\left| \sum\limits_{p=0}^Ka_p|k+p-s_1|^{H(t_0)-1/2} \right|\left| \sum\limits_{p'=0}^Ka_{p'}|k'+p'-s_1|^{H(t_0)-1/2} \right| ds_1.
\end{align*}
Let $s_2=s_1-k$, then
\begin{align*}
I_4&=n^{-\alpha H(t_0)}\int_\mathbb{R}\left| \sum\limits_{p=0}^Ka_p|p-s_2|^{H(t_0)-1/2} \right|\left| \sum\limits_{p'=0}^Ka_{p'}|p'-(s_2+k-k')|^{H(t_0)-1/2} \right| ds_2.
\end{align*}
Similar to the proof of Lemma 3.6 in \cite{Jacques2012}, we can prove that there exists $K>0$ such that for $|k-k'|\geq K$, one has 
\begin{align}\label{eq.38}
I_4&\leq C n^{-2 H(t_0)}|k-k'|^{ H(t_0)-(L+1)}.
\end{align}
Now we work with $I_2$. Applying Cauchy-Schwarz's inequality, then
\begin{align*}
I_2&\leq \left(\int_\mathbb{R}|f(k,n,s)-g(k,n,s)|^2 ds\int_\mathbb{R}|g(k',n,s)|^2 ds|\right)^{1/2}.
\end{align*}
Moreover, from (\ref{eq.27}) and (\ref{eq.31}),
Moreover, we have
\begin{align*}
\int_\mathbb{R}|g(k',n,s)|^2 ds
&=M_{t_0}^2 n^{-2H(t_0)}\\
\int_\mathbb{R}|f(k,n,s)-g(k,n,s)|^2 ds&\leq Cn^{-2\gamma}
\end{align*}
where $M_{t_0}$ is defined by (\ref{eq.24}), $C$ is a constant.\\
Thus 
\begin{align}\label{eq.39}
I_2&\leq Cn^{-(H(t_0)+\gamma)}
\end{align}
Similarly, we also get 
\begin{align}\label{eq.40}
I_3\leq Cn^{-(H(t_0)+\gamma)}.
\end{align}
 For $I_1$, applying Cauchy-Schwarz's inequality, then
\begin{align}
I_1&\leq \left(\int_\mathbb{R}|f(k,n,s)-g(k,n,s)|^2 ds\int_\mathbb{R}|f(k',n,s)-g(k',n,s)|^2 ds|\right)^{1/2}\leq Cn^{-2\gamma}.\label{eq.41}
\end{align}
Combining with (\ref{eq.38}), (\ref{eq.39}), (\ref{eq.40}), one gets
\begin{align}\label{eq.42}
I\leq C\left(n^{-2\gamma}+n^{-(H(t_0)+\gamma)}+n^{-2 H(t_0)}|k-k'|^{H(t_0)-(L+1)}\right).
\end{align}
On the other hand, let
\begin{align}\label{eq.43}
\rho_{k,k'}&=cov\left(\frac{\triangle_{k,n}X}{||\triangle_{k,n}X||_2},\frac{\triangle_{k',n}X}{||\triangle_{k',n}X||_2}\right).
\end{align}
Since $X(t)$ is a centered Gaussian variable (see e.g. \cite{Cohen2013}), then 
\begin{align}\label{eq.44}
|\rho_{k,k'}|&=\frac{\left|\int\limits_{\mathbb{R}}f(k,n,s)f(k',n,s)ds\right|}{\left(\int\limits_{\mathbb{R}}|f(k,n,s)|^2ds\int\limits_{\mathbb{R}}|f(k',n,s)|^2ds\right)^{1/2}} \notag\\
&\leq\frac{\int\limits_{\mathbb{R}}|f(k,n,s)f(k',n,s)|ds}{\left(\int\limits_{\mathbb{R}}|f(k,n,s)|^2ds\int\limits_{\mathbb{R}}|f(k',n,s)|^2ds\right)^{1/2}}.
\end{align}
However, from Lemma \ref{lem.1.18} and Remark \ref{remark.1.3}, for $n\geq n_0, |k- k'|\geq K$ and $M_1<\sigma_{k,n}(u)<M_2$, using (\ref{eq.42}) and the fact that $\int\limits_\mathbb{R}|f(k,n,s)|^2ds=n^{-2 H(u)}\sigma_{k,n}(t_0)^2$, it follows that
\begin{align}\label{eq.45}
|\rho_{k,k'}|&\leq \frac{C\left(n^{-2\gamma}+n^{-(H(t_0)+\gamma)}+n^{-2H(t_0)}|k-k'|^{ H(t_0)-(L+1)}\right)}{n^{-2 H(t_0)}} \notag\\
&=C\left(n^{2(H(t_0)-\gamma)}+n^{H(t_0)-\gamma)}+|k-k'|^{H(t_0)-(L+1)}\right).
\end{align}
Since $H(t_0)<\gamma<1$, then there exist $n_1,k_0\in\mathbb{N}, n_1\geq n_0$ and $\rho_*$ such that for $n\geq n_1, |k-k'|> k_0$, we have 
\begin{align*}
|\rho_{k,k'}|\leq\rho_*<1.
\end{align*} 
Together with Lemma A.1 in \cite{DanIst2017}, there exist a constant $C$ such that 
\begin{align*}
\left|cov\left(\left|\frac{\triangle_{k,n}X}{||\triangle_{k,n}X||_2}\right|^\beta,\left|\frac{\triangle_{k',n}X}{||\triangle_{k',n}X||_2}\right|^\beta\right)\right|&\leq C \rho_{k,k'}^2\\
& \leq C\left(\int_\mathbb{R} \frac{|f(k,n,s)f(k',n,s)|}{||\triangle_{k,n}X||_\alpha||\triangle_{k',n}X||_2}ds\right)^2.
\end{align*}
Thus, one gets
\begin{align*}
\left|cov(|\triangle_{k,n}X|^\beta,|\triangle_{k',n}X|^\beta)\right|&\leq C\left(\int_\mathbb{R}|f(k,n,s)|^2 ds\right)^{\frac{\beta}{2}-1}\left(\int_\mathbb{R}|f(k',n,s)|^2 ds\right)^{\frac{\beta}{2}-1}\\
&\times \left(\int_\mathbb{R}\left|f(k,n,s)f(k',n,s)\right|ds\right)^2.
\end{align*}
Moreover, from Remark \ref{remark.1.3}, there exists  $M_1>0$ such that $ M_1<\sigma_{k,n}(u), M_1<\sigma_{k',n}(u)$. Then using the fact that $\frac{\beta}{2}-1<0$ and 
\begin{align*}
\int\limits_\mathbb{R}|f(k,n,s)|^2 ds&=n^{-2 H(t_0)}\sigma_{k,n}(t_0)^2> M_1^{2}n^{-2 H(t_0)},\\
\int\limits_\mathbb{R}|f(k',n,s)|^2 ds&=n^{-2 H(t_0)}\sigma_{k',n}(t_0)^2> M_1^{2}n^{-2 H(t_0)}, 
\end{align*} 
combining with (\ref{eq.42})
\begin{align*}
I^2&\leq  C\left(n^{-4\gamma}+n^{-2(H(t_0)+\gamma)}+n^{-4 H(t_0)}|k-k'|^{2H(t_0)-2(L+1)}\right),
\end{align*}
 it follows that 
\begin{align*}
&\left|cov(|\triangle_{k,n}X|^\beta,|\triangle_{k',n}X|^\beta)\right|\\
&\leq C\left(M_1^2 n^{-2 H(t_0)}\right)^{2\left(\frac{\beta}{2}-1\right)}\left(\int_\mathbb{R}\left|f(k,n,s)f(k',n,s)\right|ds\right)^2\\
&\leq C n^{-2\beta H(t_0)+4 H(t_0)}\left(n^{-4\gamma}+n^{-2(H(t_0)+\gamma)}+n^{-4 H(t_0)}|k-k'|^{2H(t_0)-2(L+1)}\right)\\
&= C n^{-2\beta H(t_0)}\left(n^{4(H(t_0)-\gamma)}+n^{2(H(t_0)-\gamma)}+|k-k'|^{2H(t_0)-2(L+1)}\right).
\end{align*} 
One gets the conclusion.
\end{proof}
Now we come back to the proof of Theorem \ref{thm.1}.
\begin{proof}[\emph{\bf{Proof of Theorem \ref{thm.1}.}}]
\begin{enumerate}
\item One gets
\begin{align*}
\mathbb{E}(W_{t_0,n}(\beta)-M_{t_0,\beta})^2&= \mathbb{E}W_{t_0,n}^2(\beta)-\left(\mathbb{E}W_{t_0,n}(\beta)\right)^2+\left(\mathbb{E}W_{t_0,n}(\beta)-M_{t_0,\beta} \right)^2.
\end{align*}
Applying Lemma \ref{lem.1.20} and Lemma \ref{lem.1.22}, for $n\geq n_1$, one has
\begin{align*}
\begin{split}
&\mathbb{E}W_{t_0,n}^2(\beta)-\left(\mathbb{E}W_{t_0,n}(\beta)\right)^2\\
&=\frac{1}{\upsilon_{\gamma,n}(t_0)^2}\sum\limits_{k,k'\in \nu_{\gamma,n}(t_0)}cov(\left|\frac{\triangle_{k,n}X}{n^{-H(t_0)}}\right|^\beta,\left|\frac{\triangle_{k',n}X}{n^{-H(t_0)}}\right|^\beta)\\
&\leq \frac{1}{\upsilon_{\gamma,n}(t_0)^2}\sum\limits_{k,k'\in \nu_{\gamma,n}(t_0), |k-k'|\leq k_0}\left|cov\left(\left|\frac{\triangle_{k,n}X}{n^{-H(t_0)}}\right|^\beta,\left|\frac{\triangle_{k',n}X}{n^{-H(t_0)}}\right|^\beta\right)\right|\\
&+\frac{n^{2\beta H(t_0)}}{\upsilon_{\gamma,n}(t_0)^2}\sum\limits_{k,k'\in \nu_{\gamma,n}(t_0),|k-k'|> k_0}\left|cov\left(\left|\triangle_{k,n}X|^\beta,|\triangle_{k',n}X\right|^\beta\right)\right|\\ 
&\leq \frac{1}{\upsilon_{\gamma,n}(t_0)^2}\sum\limits_{|p|\leq k_0, p\in\mathbb{Z}}(\upsilon_{\gamma,n}(t_0)-|p|)C+ \frac{n^{2\beta H(t_0)}}{\upsilon_{\gamma,n}(t_0)^2}\\
&\times \sum\limits_{ k_0< |p|\leq \upsilon_{\gamma,n}(t_0),p\in\mathbb{Z}} (\upsilon_{\gamma,n}(t_0)-|p|) C n^{-2\beta H(t_0)}\left(n^{4(H(t_0)-\gamma)}+n^{2(H(t_0)-\gamma)}+|p|^{2H(t_0)-2(L+1)}\right) \\
 &\leq \frac{1}{\upsilon_{\gamma,n}(t_0)^2}\sum\limits_{|p|\leq k_0,p\in\mathbb{Z}}(\upsilon_{\gamma,n}(t_0)-|p|)C\\
 &+\frac{C}{\upsilon_{\gamma,n}(t_0)}\sum\limits_{ k_0< |p|\leq \upsilon_{\gamma,n}(t_0),p\in\mathbb{Z}} \left(n^{4(H(t_0)-\gamma)}+n^{2(H(t_0)-\gamma)}+|p|^{2H(t_0)-2(L+1)}\right).
\end{split} 
 \end{align*}
 Combining with Lemma \ref{lem.1.19} and using the fact that $\upsilon_{\gamma,n}(t_0)=[2n^{1-\gamma}-K] $ or $[2n^{1-\gamma}-K]+1 $ depending on the parity of $[2n^{1-\gamma}-K]$, it follows that
 \begin{align*}
  \mathbb{E}(W_{t_0,n}(\beta)-M_{t_0,\beta})^2&\leq C\left(n^{4 (H(t_0)-\gamma)\wedge 2(H(t_0)-\gamma)}+ n^{\gamma-1}+n^{4(H(t_0)-\gamma)}+n^{2(H(t_0)-\gamma)}\right)\\
  &+\frac{C}{\upsilon_{\gamma,n}(t_0)}\sum\limits_{|p|\leq \upsilon_{\gamma,n}(t_0)}|p|^{2H(t_0)-2(L+1)}.
  \end{align*}
  Since $2H(t_0)-2(L+1)<-1$, from Lemma A.4 in \cite{DanIst2017}, one gets 
   $$\frac{1}{\upsilon_{\gamma,n}(t_0)}\sum\limits_{|p|\leq \upsilon_{\gamma,n}(t_0)p\in\mathbb{Z}}|p|^{2H(t_0)-2(L+1)}=O(\upsilon_{\gamma,n}(t_0)^{-1})=O(n^{\gamma-1}).$$
  Moreover $4 (H(t_0)-\gamma)\wedge 2(H(t_0)-\gamma)\leq 2(H(t_0)-\gamma)<0$ and $4(H(t_0)-\gamma)<2(H(t_0)-\gamma)$, then we can write
  \begin{align}\label{eq.46}
   \mathbb{E}(W_{t_0,n}(\beta)-M_{t_0,\beta})^2&\leq C \left(n^{\gamma-1}+n^{2(H(t_0)-\gamma)} \right)\notag\\
   &\leq C d_n^2
  \end{align}
  where $d_n$ is defined by (\ref{eq.15}).\\
     For $\epsilon >0$ fixed, using Markov's inequality, we get
 \begin{align}\label{eq.47}
 \mathbb{P}(|W_{t_0,n}(\beta)-M_{t_0,\beta}|>\epsilon)&\leq \frac{\mathbb{E}(W_{t_0,n}(\beta)-M_{t_0,\beta})^2}{\epsilon^2}\leq \frac{Cd_n^2}{\epsilon}.
 \end{align}
 Since $\lim\limits_{n\to+\infty}d_n=0$, it follows that
 \begin{align*}
\lim\limits_{n\to +\infty} W_{t_0,n}(\beta)&\stackrel{(\mathbb{P})}{=} M_{t_0,\beta}.
\end{align*}
From (\ref{eq.47}), one also gets
$$W_{t_0,n}(\beta)-M_{t_0,\beta}=O_\mathbb{P}(d_n).$$
\item 
Let $\phi: \mathbb{R}_+\times \mathbb{R}_+\rightarrow \mathbb{R}$ be defined by
\begin{align}\label{eq.48}
\phi(x,y)&=\frac{1}{\beta}\log_2\frac{x}{y}.
\end{align}
From the proof of the latter part, then one gets 
$$W_{t_0,n}(\beta)\stackrel{\mathbb{P}}{\rightarrow}M_{t_0,\beta},W_{t_0,n/2}(\beta)\stackrel{\mathbb{P}}{\rightarrow}M_{t_0,\beta} $$
 as $n\to+\infty$. \\
It follows that $(W_{t_0,n/2}(\beta),W_{t_0,n}(\beta))\stackrel{\mathbb{P}}{\rightarrow}(M_{t_0,\beta},M_{t_0,\beta})$ as $n\to+\infty$.
Since $\phi$ is continuous, using the continuous mapping theorem, it induces
\begin{align*}
\phi(W_{t_0,n/2}(\beta),W_{t_0,n}(\beta))&=\frac{1}{\beta}\log_2\frac{W_{t_0,n/2}(\beta)}{W_{t_0,n}(\beta)}\\
&=\frac{1}{\beta}\log_2\frac{V_{t_0,n/2}(\beta)}{V_{t_0,n}(\beta)}-H(t_0)\stackrel{\mathbb{P}}{\rightarrow}\phi(M_{t_0,\beta},M_{t_0,\beta})=0
\end{align*} 
  as $n\to+\infty$. Then $\lim\limits_{n\to+\infty}\widehat{H}_n(t_0)=H(t_0)$.\\
  Moreover, since $W_{t_0,n}(\beta)-M_{t_0,\beta}=O_\mathbb{P}(d_n),W_{t_0,n/2}(\beta)-M_{t_0,\beta}=O_\mathbb{P}(d_{n/2})= O_\mathbb{P}(d_n)$ and $\phi$ is differentiable at $(M_{t_0,\beta},M_{t_0,\beta})$, applying Lemma 4.10 in \cite{DanIst2017} 
  , one gets 
  $\widehat{H}_n-H(t_0)=O_\mathbb{P}(d_n).$
%
\end{enumerate}
  \end{proof}
\subsection{Proof of Theorem \ref{thm.2}}\label{Proof of Theorem 2}
We first present the following lemma.
\begin{lem}\label{lem.1.21}
There exist $n_1,k_0\in\mathbb{N}$ and $0<\eta<1$ such that for all $n\geq n_1, k,k'\in\nu_\gamma(t_0), |k-k'|>k_0$, we have 
\begin{align}
cov(|\triangle_{k,n}X|^\beta,|\triangle_{k',n}X|^\beta)&\leq C^*(\eta)n^{-2\beta H(t_0)+\alpha H(t_0)}\notag\\
&\times \left( n^{-\alpha \gamma}+n^{-\frac{\alpha(H(t_0)+\gamma)}{2}}+n^{-\alpha H(t_0)}|k-k'|^{\frac{\alpha H(t_0)-\alpha(L+1)}{2}}\right),\label{eq.49}
\end{align}
where $C^*(\eta)$ is a constant depending on $\eta$.
\end{lem}
\begin{proof}
Let  
\begin{align*}
I&=\int\limits_\mathbb{R}\left|f(k,n,s)f(k',n,s)\right|^{\frac{\alpha}{2}}ds,
\end{align*}
where $f(k,n,s)$ is defined by (\ref{eq.21})
 Since $0<\alpha/2<1$, following Lemma 2.7.13 in \cite{Taqqu1994}, one gets
 \begin{align*}
 |f(k,n,s)|^{\alpha/2}&\leq |f(k,n,s)-g(k,n,s)|^{\alpha/2}+|g(k,n,s)|^{\alpha/2},\\
 |f(k',n,s)|^{\alpha/2}&\leq |f(k',n,s)-g(k',n,s)|^{\alpha/2}+|g(k',n,s)|^{\alpha/2}.
 \end{align*}
 It follows that
\begin{align}
I&\leq \int\limits_\mathbb{R}\left(|f(k,n,s)-g(k,n,s)|^{\alpha/2}+|g(k,n,s)|^{\alpha/2}\right)\notag\\
&\times\left(|f(k',n,s)-g(k',n,s)|^{\alpha/2}+|g(k',n,s)|^{\alpha/2}\right)ds \notag\\
&:=I_1+I_2+I_3+I_4,\label{eq.50}
\end{align}
where
\begin{align*}
I_1&=\int\limits_\mathbb{R}|f(k,n,s)-g(k,n,s)|^{\alpha/2}|f(k',n,s)-g(k',n,s)|^{\alpha/2}ds\\
I_2&=\int\limits_\mathbb{R}|f(k,n,s)-g(k,n,s)|^{\alpha/2}|g(k',n,s)|^{\alpha/2}ds\\
I_3&=\int\limits_\mathbb{R}|f(k',n,s)-g(k',n,s)|^{\alpha/2}|f(k,n,s)|^{\alpha/2}ds\\
I_4&=\int\limits_\mathbb{R}|g(k,n,s)|^{\alpha/2}|g(k',n,s)|^{\alpha/2}ds.
\end{align*}
Now we will find the bound for each term $I_1, I_2,I_3, I_4$. 
For $I_4$, we have 
\begin{align*}
I_4&=\int\limits_\mathbb{R}|g(k,n,s)|^{\alpha/2}|g(k',n,s)|^{\alpha/2}ds\\
&=\int\limits_\mathbb{R}\left| \sum\limits_{p=0}^Ka_p|\frac{k+p}{n}-s|^{H(u)-1/\alpha} \right|^{\alpha/2} \left| \sum\limits_{p'=0}^Ka_{p'}|\frac{k'+p'}{n}-s|^{H(u)-1/\alpha} \right|^{\alpha/2} ds.
\end{align*}
Making the change of variable $s=s_1/n$, it induces
\begin{align*}
I_4&=n^{-\alpha H(t_0)}\int\limits_\mathbb{R}\left| \sum\limits_{p=0}^Ka_p|k+p-s_1|^{H(t_0)-1/\alpha} \right|^{\alpha/2}\left| \sum\limits_{p'=0}^Ka_{p'}|k'+p'-s_1|^{H(u)-1/\alpha} \right|^{\alpha/2} ds_1.
\end{align*}
Let $s_2=s_1-k$, then
\begin{align*}
I_4&=n^{-\alpha H(t_0)}\int\limits_\mathbb{R}\left| \sum\limits_{p=0}^Ka_p|p-s_2|^{H(t_0)-1/\alpha} \right|^{\alpha/2}\left| \sum\limits_{p'=0}^Ka_{p'}|p'-(s_2+k-k'|^{H(t_0)-1/\alpha} \right|^{\alpha/2} ds_2.
\end{align*}
Following Lemma 3.6 in \cite{Jacques2012}, then there exists $K>0$ such that for $|k-k'|\geq K$, one has
\begin{align}\label{eq.51}
I_4&\leq C n^{-\alpha H(t_0)}|k-k'|^{\frac{\alpha H(t_0)-\alpha(L+1)}{2}}.
\end{align}
To find a bound for $I_2$, one can apply Cauchy-Schwarz's inequality and obtains
\begin{align*}
I_2&\leq \left(\int\limits_\mathbb{R}|f(k,n,s)-g(k,n,s)|^\alpha ds\int\limits_\mathbb{R}|g(k',n,s)|^\alpha ds|\right)^{1/2}.
\end{align*}
Moreover, from (\ref{eq.27}) and (\ref{eq.31}), we have
\begin{align*}
\int\limits_\mathbb{R}|g(k',n,s)|^\alpha ds
&=M_{t_0}^\alpha n^{-\alpha H(t_0)}.\\
\int\limits_\mathbb{R}|f(k,n,s)-g(k,n,s)|^\alpha ds&\leq Cn^{-\alpha\gamma}.
\end{align*}
Thus 
\begin{align}\label{eq.52}
I_2&\leq Cn^{-\frac{\alpha(H(t_0)+\gamma)}{2}}
\end{align}
Similarly, one also gets 
\begin{align}\label{eq.53}
I_3\leq Cn^{-\frac{\alpha(H(t_0)+\gamma)}{2}}.
\end{align}
 For $I_1$, applying Cauchy-Schwartz's inequality, then
\begin{align}
I_1&\leq \left(\int\limits_\mathbb{R}|f(k,n,s)-g(k,n,s)|^\alpha ds\int\limits_\mathbb{R}|f(k',n,s)-g(k',n,s)|^\alpha ds|\right)^{1/2}\leq Cn^{-\gamma\alpha}.\label{eq.54}
\end{align}
Combining with (\ref{eq.51}), (\ref{eq.52}), (\ref{eq.52}), one gets
\begin{align}\label{eq.55}
I\leq C\left(n^{-\alpha\gamma}+n^{-\frac{\alpha(H(t_0)+\gamma)}{2}}+n^{-\alpha H(t_0)}|k-k'|^{\frac{\alpha H(t_0)-\alpha(L+1)}{2}}\right).
\end{align}
Let 
  \begin{align}
  \eta_{k,k'}&=\left[\frac{\triangle_{k,n}X}{||\triangle_{k,n}X||_\alpha},\frac{\triangle_{k',n}X}{||\triangle_{k',n}X||_\alpha}\right]_2,\label{eq.56}
 \end{align}
 where $$\left[\int\limits_\mathbb{R}f(s)M_\alpha(ds), \int\limits_\mathbb{R}g(s)M_\alpha(ds)\right]_2=\int\limits_\mathbb{R}|f(s)g(s)|^{\alpha/2}ds.$$ 
One has
\begin{align*}
\eta_{k,k'}&=\frac{\int\limits_{\mathbb{R}}|f(k,n,s)f(k',n,s)|^{\alpha/2}ds}{\left(\int\limits_{\mathbb{R}}|f(k,n,s)|^{\alpha}ds\int\limits_{\mathbb{R}}|f(k',n,s)|^{\alpha}\right)^{1/2}ds}.
\end{align*}
For $n\geq n_0$, from Lemma \ref{lem.1.18} then
$$M_1<\sigma_{k,n}(t_0)<M_2, \int\limits_{\mathbb{R}}|f(k,n,s)|^\alpha ds =n^{-\alpha H(t_0)}\sigma_{k,n}^\alpha(u). $$
 Combining with (\ref{eq.55}), it follows that 
\begin{align*}
\eta_{k,k'}&\leq C\frac{n^{-\alpha\gamma}+n^{-\frac{\alpha(H(t_0)+\gamma)}{2}}+n^{\alpha H(t_0)}|k-k'|^{\frac{\alpha H(t_0)-\alpha(L+1)}{2}}}{n^{-\alpha H(t_0)}}\\
&=C\left(n^{\alpha(H(t_0)-\gamma)}+n^{\frac{\alpha(H(t_0)-\gamma)}{2}}+|k-k'|^{\frac{\alpha H(t_0)-\alpha(L+1)}{2}}\right).
\end{align*}
Since $H(t_0)<\gamma<1$ and $\frac{\alpha H(t_0)-\alpha (L+1)}{2}$ then there exist $n_1,k_0\in\mathbb{N}, n_1\geq n_0$ and $\eta$ such that for $n\geq n_1, |k-k'|> k_0$, we have 
$
0<\eta_{k,k'}\leq\eta<1.
$ Following Theorem 4.2 in \cite{DanIst2017}, one has
\begin{align*}
\left|cov\left(\left|\frac{\triangle_{k,n}X}{||\triangle_{k,n}X||_\alpha}\right|^\beta,\left|\frac{\triangle_{k',n}X}{||\triangle_{k',n}X||_\alpha}\right|^\beta\right)\right|&\leq C(\eta)\int\limits_\mathbb{R} \left|\frac{f(k,n,s)f(k',n,s}{||\triangle_{k,n}X||_\alpha||\triangle_{k',n}X||_\alpha}\right|^{\frac{\alpha}{2}}ds,
\end{align*}
where $C(\eta)$ is a constant depending on $\eta$.
Thus, one obtains
\begin{align*}
\left|cov(|\triangle_{k,n}X|^\beta,|\triangle_{k',n}X|^\beta)\right|&\leq C(\eta)\left(\int\limits_\mathbb{R}|f(k,n,s)|^\alpha ds\right)^{\frac{\beta}{\alpha}-\frac{1}{2}}\left(\int\limits_\mathbb{R}|f(k',n,s)|^\alpha ds\right)^{\frac{\beta}{\alpha}-\frac{1}{2}}\\
&\times\int\limits_\mathbb{R}\left|f(k,n,s)f(k',n,s)\right|^{\frac{\alpha}{2}}ds.
\end{align*}
Moreover, from (\ref{eq.55}) and Lemma \ref{lem.1.18}, using the fact that $ \frac{\beta}{\alpha}-\frac{1}{2}<0 $ and there exists $M_1>0$ such that 
$$M_1<\sigma_{k,n}(t_0), M_1<\sigma_{k',n}(t_0),$$  one has
\begin{align*}
\int\limits_\mathbb{R}|f(k,n,s)|^\alpha ds&=n^{-\alpha H(t_0)}\sigma_{k,n}^\alpha(t_0)> M_1^{\alpha}n^{-\alpha H(t_0)},\\
\int\limits_\mathbb{R}|f(k',n,s)|^\alpha ds&=n^{-\alpha H(t_0)}\sigma_{k',n}^\alpha(t_0)> M_1^{\alpha}n^{-\alpha H(t_0)}.
\end{align*}
 It follows that
\begin{align*}
\left|cov(|\triangle_{k,n}X|^\beta,|\triangle_{k',n}X|^\beta)\right|&\leq C^*(\eta)\left(M_1^\alpha n^{-\alpha H(t_0)}\right)^{2\left(\frac{\beta}{\alpha}-\frac{1}{2}\right)}\int\limits_\mathbb{R}\left|f(k,n,s)f(k',n,s)\right|^{\frac{\alpha}{2}}ds\\
&\leq C^*(\eta)n^{-2\beta H(t_0)+\alpha H(t_0)}\\
&\times \left(  n^{-\alpha\gamma}+n^{-\frac{\alpha(H(t_0)+\gamma)}{2}}+n^{-\alpha H(t_0)}|k-k'|^{\frac{\alpha H(t_0)-2\alpha}{2}} \right), 
\end{align*}
where $C^*(\eta)$ is a constant depending on $\eta$.
\end{proof}
\begin{proof}[\emph{\bf{Proof of Theorem \ref{thm.2}.}}]
\begin{enumerate}
\item One gets
\begin{align*}
\mathbb{E}(W_{t_0,n}(\beta)-M_{t_0,\beta})^2&= \mathbb{E}W_{t_0,n}^2(\beta)-\left(\mathbb{E}W_{t_0,n}(\beta)\right)^2+\left(\mathbb{E}W_{t_0,n}(\beta)-M_{t_0,\beta} \right)^2.
\end{align*}
Applying Lemma \ref{lem.1.20} and Lemma \ref{lem.1.21}, for $n\geq n_1$ where $n_1$ and $k_0$ are defined as in Lemma \ref{lem.1.21},  we have
\begin{align*}
\begin{split}
&\mathbb{E}W_{t_0,n}^2(\beta)-\left(\mathbb{E}W_{t_0,n}(\beta)\right)^2\\
&=\frac{1}{\upsilon_{\gamma,n}(t_0)^2}\sum\limits_{k,k'\in \nu_{\gamma,n}(t_0)}cov(\left|\frac{\triangle_{k,n}X}{n^{-H(t_0)}}\right|^\beta,\left|\frac{\triangle_{k',n}X}{n^{-H(t_0)}}\right|^\beta)\\
&\leq \frac{1}{\upsilon_{\gamma,n}(t_0)^2}\sum\limits_{k,k'\in \nu_{\gamma,n}(t_0), |k-k'|\leq k_0}\left|cov\left(\left|\frac{\triangle_{k,n}X}{n^{-H(t_0)}}\right|^\beta,\left|\frac{\triangle_{k',n}X}{n^{-H(t_0)}}\right|^\beta\right)\right|\\
&+\frac{n^{2\beta H(t_0)}}{\upsilon_{\gamma,n}(t_0)^2}\sum\limits_{k,k'\in \nu_{\gamma,n}(t_0),|k-k'|> k_0}\left|cov\left(\left|\triangle_{k,n}X|^\beta,|\triangle_{k',n}X\right|^\beta\right)\right|\\
&\leq \frac{1}{\upsilon_{\gamma,n}(t_0)^2}\sum\limits_{|p|\leq k_0, p\in\mathbb{Z}}(\upsilon_{\gamma,n}(t_0)-|p|)C\\
&+ \frac{n^{2\beta H(t_0)}}{\upsilon_{\gamma,n}(t_0)^2}
 \sum\limits_{ k_0< |p|\leq \upsilon_{\gamma,n}(t_0), p\in\mathbb{Z}} (\upsilon_{\gamma,n}(t_0)-|p|) C^*(\eta) n^{-2\beta H(t_0)+\alpha H(t_0)}\\
 &\times \left(n^{-\alpha\gamma}+n^{-\frac{\alpha(H(t_0)+\gamma)}{2}}+n^{-\alpha H(t_0)}|p|^{\frac{\alpha H(t_0)-\alpha(L+1)}{2}}\right) \\
 &\leq \frac{1}{\upsilon_{\gamma,n}(t_0)^2}\sum\limits_{|p|\leq k_0, p\in\mathbb{Z}}(\upsilon_{\gamma,n}(t_0)-|p|)C\\
 &+\frac{C^*(\eta)}{\upsilon_{\gamma,n}(t_0)}\sum\limits_{ k_0< |p|\leq \upsilon_{\gamma,n}(t_0),p\in\mathbb{Z}} \left( n^{\alpha(H(t_0)-\gamma)}+n^{\frac{\alpha(H(t_0)-\gamma)}{2}}+|p|^{\frac{\alpha H(t_0)-\alpha(L+1)}{2}} \right).
 \end{split}
 \end{align*}
 Combining with Lemma \ref{lem.1.19} and the fact that $\upsilon_{\gamma,n}(t_0)=[2n^{1-\gamma}-K] $ or $[2n^{1-\gamma}-K]+1 $ depending on the parity of $[2n^{1-\gamma}-K]$, it follows that
 \begin{align*}
  \mathbb{E}(W_{t_0,n}(\beta)-M_{t_0,\beta})^2&
  \leq C\left(n^{2\alpha (H(t_0)-\gamma)\wedge 2(H(t_0)-\gamma)}+ n^{\gamma-1}+n^{\alpha(H(t_0)-\gamma)}+n^{\frac{\alpha(H(t_0)-\gamma)}{2}}\right)\\
  &+\frac{C}{\upsilon_{\gamma,n}(t_0)}\sum\limits_{|p|\leq \upsilon_{\gamma,n}(t_0)}|p|^{\frac{\alpha H(t_0)-\alpha(L+1)}{2}}. 
   \end{align*}
 From Lemma A.4 in \cite{DanIst2017}, one gets
\begin{align*}
   &\frac{1}{\upsilon_{\gamma,n}(t_0)}\sum\limits_{|p|\leq \upsilon_{\gamma,n}(t_0),p\in\mathbb{Z}}|p|^{\frac{\alpha H(t_0)-\alpha (L+1)}{2}}\\
   &=\begin{cases}
   O\left(\upsilon_{\gamma,n}(t_0)^{-1}\right)=O\left(n^{\gamma-1}\right)& \text{ if  } \frac{\alpha H(t_0)-\alpha(L+1)}{2}<-1,\\
    O\left(\upsilon_{\gamma,n}(t_0)^ \frac{\alpha H(t_0)-\alpha(L+1)}{2}\right)=O\left(n^{\frac{(1-\gamma)(\alpha H(t_0)-\alpha(L+1))}{2}}\right)& \text{ if  }  -1<\frac{\alpha H(t_0)-\alpha(L+1)}{2}<0,\\
     O\left(\frac{\ln(\upsilon_{\gamma,n}(t_0))}{\upsilon_{\gamma,n}(t_0)}\right)=O\left(n^{\gamma-1}\ln n\right)& \text{ if  }  \frac{\alpha H(t_0)-\alpha(L+1)}{2}=-1.
   \end{cases}
    \end{align*} 
  Combining with the fact that $2\alpha(H(t_0)-\gamma)\wedge 2(H(t_0)-\gamma)\leq \frac{\alpha(H(t_0)-\gamma)}{2}<0$ and $\alpha(H(t_0)-\gamma)<\frac{\alpha(H(t_0)-\gamma)}{2}<0$, with $d_n$ defined by (\ref{eq.18}), we can write
  \begin{align}\label{eq.57}
   \mathbb{E}(W_{t_0,n}(\beta)-M_{t_0,\beta})^2&\leq C \left(n^{\gamma-1}+n^{\frac{\alpha(H(t_0)-\gamma)}{2}}+ \frac{1}{\upsilon_{\gamma,n}(t_0)}\sum\limits_{|p|\leq \upsilon_{\gamma,n}(t_0),p\in\mathbb{Z}}|p|^{\frac{\alpha H(t_0)-\alpha (L+1)}{2}}\right)\notag\\
   &\leq C d_n^2.
  \end{align}
   It is obvious that $\lim\limits_{n\to+\infty}d_n=0$.
     For $\epsilon >0$ fixed, using Markov's inequality, we get
 \begin{align*}
 \mathbb{P}(|W_{t_0,n}(\beta)-M_{t_0,\beta}|>\epsilon)&\leq \frac{\mathbb{E}(W_{t_0,n}(\beta)-M_{t_0,\beta})^2}{\epsilon^2}\leq \frac{Cd_n^2}{\epsilon^2}.
 \end{align*}
 It follows that
 \begin{align*}
\lim\limits_{n\to +\infty} W_{t_0,n}(\beta)&\stackrel{(\mathbb{P})}{=} M_{t_0,\beta}
\end{align*}
and
$$W_{t_0,n}(\beta)-M_{t_0,\beta}=O_\mathbb{P}(d_n).$$
\item One can use the same function $\phi$ as in (\ref{eq.48}) and the similar way to the proof of Theorem \ref{thm.1} to get the conclusion.
  \end{enumerate}
  \end{proof}
\subsection{Proof of Theorem \ref{thm.3}}\label{Proof of Theorem 3}
For $t_0\in U$ fixed, we will prove that $\psi_{-\beta_1,-\beta_2}(V_{t_0,n}(\beta_1),V_{t_0,n}(\beta_2)) \stackrel{\mathbb{P}}{\rightarrow} 
h_{-\beta_1,-\beta_2}(\alpha)$ as $n\to+\infty$.\\
From Lemma 4.2 in \cite{DanIst2017}, one gets
\begin{align*}
M_{t_0,\beta_1}^{\beta_2}&=\left(\frac{M_{t_0}^{\beta_1}C_{\beta_1}}{\sqrt{2\pi}}\int\limits_\mathbb{R}\frac{e^{-|y|^\alpha}}{|y|^{1+\beta_1}}dy\right)^{\beta_2}\\
&=\left(\frac{M_{t_0}^{\beta_1}2^{\beta_1}\Gamma(\frac{\beta_1+1}{2})\Gamma(1-\frac{\beta_1}{\alpha})}{\sqrt{\pi}\Gamma(1-\frac{\beta_1}{2})}\right)^{\beta_2}\\
&=M_{t_0}^{\beta_1\beta_2}2^{\beta_1\beta_2}\left(\frac{\Gamma(\frac{\beta_1+1}{2})\Gamma(1-\frac{\beta_1}{\alpha})}{\sqrt{\pi}\Gamma(1-\frac{\beta_1}{2})}\right)^{\beta_2}
\end{align*}
where $M_{t_0}, M_{t_0,\beta}, C_\beta$ are defined by (\ref{eq.24}), (\ref{eq.34}) and (\ref{eq.35}), respectively.\\
Similarly, one also has
\begin{align*}
M_{t_0,\beta_2}^{\beta_1}&=M_{t_0}^{\beta_1\beta_2}2^{\beta_1\beta_2}\left(\frac{\Gamma(\frac{\beta_2+1}{2})\Gamma(1-\frac{\beta_2}{\alpha})}{\sqrt{\pi}\Gamma(1-\frac{\beta_2}{2})}\right)^{\beta_1}.
\end{align*}
Then
\begin{align*}
\frac{M_{t_0,\beta_1}^{\beta_2}}{M_{t_0,\beta_2}^{\beta_1}}&=\frac{\pi^{\frac{\beta_1-\beta_2}{2}} \Gamma^{\beta_1}(1-\frac{\beta_2}{2}) \Gamma^{\beta_2}(\frac{\beta_1+1}{2}) \Gamma^{\beta_2}(1-\frac{\beta_1}{\alpha})}{\Gamma^{\beta_2}(1-\frac{\beta_1}{2})
 \Gamma^{\beta_1}(\frac{\beta_2+1}{2})\Gamma^{\beta_1}(1-\frac{\beta_2}{\alpha})}.
\end{align*}
Taking the natural logarithm, it follows that
\begin{align*}
&\beta_2\ln(M_{t_0,\beta_1})-\beta_1\ln(M_{t_0,\beta_2})\\
&=\frac{\beta_1-\beta_2}{2}\ln (\pi)+\beta_1 \ln  \Gamma(1-\frac{\beta_2}{2})+
\beta_2 \ln \Gamma(\frac{\beta_1+1}{2})+\beta_2\ln \Gamma(1-\frac{\beta_1}{\alpha})\\
&-\beta_2 \ln \Gamma(1-\frac{\beta_1}{2})-\beta_1 \ln \Gamma(\frac{\beta_2+1}{2})
-\beta_1\ln\Gamma(1-\frac{\beta_2}{\alpha}).
\end{align*}
Thus, one gets
\begin{align*}
&\beta_2\ln \Gamma(1-\frac{\beta_1}{\alpha})- 
 \beta_1\ln\Gamma(1-\frac{\beta_2}{\alpha})\\
 &=
 \beta_2\ln(M_{u,\beta_1})-\beta_1\ln(M_{u,\beta_2})
 +\frac{\beta_2-\beta_1}{2}\ln (\pi)-
 \beta_1 \ln \Gamma(1-\frac{\beta_2}{2})-\beta_2 \ln \Gamma(\frac{\beta_1+1}{2})\\
 &+\beta_2 \ln\Gamma(1-\frac{\beta_1}{2})
 +\beta_1 \ln \Gamma(\frac{\beta_2+1}{2}).
\end{align*}
Then
$$
h_{-\beta_1,-\beta_2}(\alpha)=\psi_{-\beta_1,-\beta_2}
 (M_{t_0,\beta_1}, M_{t_0,\beta_2}),
$$
where the function $ \psi_{u,v}, h_{u,v}$ are defined by  (\ref{eq.11}) and (\ref{eq.12}) respectively.\\
From Lemma 4.11 in \cite{DanIst2017}, it follows that $h_{u,v}$ is a strictly increasing function on $(0,+\infty)$ and 
$$\lim\limits_{x\rightarrow+\infty}h_{u,v}(x)=0, \lim\limits_{x\rightarrow 0}h_{u,v}(x)=-\infty .$$ Furthermore, there exists an inverse function $$h_{u,v}^{-1}: (-\infty,0)\rightarrow (0,+\infty)$$ which is continuous and differentiable on $(-\infty, 0)$.\\
In addition, 
\begin{align*}
\psi_{-\beta_1,-\beta_2}(W_{u,n}(\beta_1),W_{u,n}(\beta_2))=\psi_{-\beta_1,-\beta_2}(V_{u,n}(\beta_1),V_{u,n}(\beta_2)).
\end{align*}
and since $\alpha \in (0,2]$, one gets
$$h_{-\beta_1,-\beta_2}(\alpha)=\psi_{-\beta_1,-\beta_2}
(M_{t_0,\beta_1}, M_{t_0,\beta_2})<0.$$
Then
\begin{align}
 \widehat{\alpha}_n-\alpha&=\varphi_{-\beta_1,-\beta_2}\left(\psi_{-\beta_1,-\beta_2}(V_{t_0,n}(\beta_1),V_{t_0,n}(\beta_2))\right)
 -h^{-1}_{-\beta_1,-\beta_2}\left(h_{-\beta_1,-\beta_2}(\alpha)\right) \notag\\
 &=\varphi_{-\beta_1,-\beta_2}\left(\psi_{\beta_1,\beta_2}(W_{t_0,n}(\beta_1),W_{t_0,n}(\beta_2))\right)
 -\varphi_{-\beta_1,-\beta_2}\left(h_{-\beta_1,-\beta_2}(\alpha) \right) \notag\\
 &=\varphi_{-\beta_1,-\beta_2}\left(\psi_{-\beta_1,-\beta_2}(W_{t_0,n}(\beta_1),W_{t_0,n}(\beta_2))\right) \notag\\
 &-\varphi_{-\beta_1,-\beta_2}(\psi_{-\beta_1,-\beta_2}(M_{t_0,\beta_1},M_{t_0,\beta_2})). \label{eq.58}
\end{align}
 Moreover, from Theorem \ref{thm.1} and Theorem \ref{thm.2}, we have   
$
W_{t_0,n}(\beta_1)\xrightarrow{\mathbb{P}} M_{t_0,\beta_1},
 W_{t_0,n}(\beta_2)\xrightarrow{\mathbb{P}} M_{t_0,\beta_2}
 $. It follows that  
 $$(W_{t_0,n}(\beta_1),W_{t_0,n}(\beta_2))\xrightarrow{\mathbb{P}} 
 (M_{t_0,\beta_1},M_{t_0,\beta_2})$$
 as $n\rightarrow +\infty$. It is obvious that $\psi_{-\beta_1,-\beta_2}$ is a continuous function. Applying the continuous mapping theorem, it induces that
 $$\lim\limits_{n\rightarrow+\infty}\widehat{\alpha}_n\stackrel{\mathbb{P}}{=}\alpha$$.
 We also have 
 $$ W_{t_0,n}(\beta_1)- M_{t_0,\beta_1}=O_{\mathbb{P}(d_n)},
 W_{t_0,n}(\beta_2)- M_{t_0,\beta_2}=O_{\mathbb{P}(d_n)}
 $$
 where $d_n$ is defined by (\ref{eq.15}) for multifractional Brownian motion and by (\ref{eq.18}) for linear multifractional stable motion. Combining with (\ref{eq.58}) and the fact that $\varphi_{-\beta_1,-\beta_2}\circ \psi_{-\beta_1,-\beta_2}$ is differentiable at $ 
(M_{t_0,\beta_1},M_{t_0,\beta_2}),$  
we apply Lemma 4.10 in \cite{DanIst2017} and get the conclusion.
\section*{Acknowledgements}
We would like to thank Jacques Istas for his helpful comments. 
 %
%
\bibliographystyle{plain}
\bibliography{References}

\begin{thebibliography}{10}

\bibitem{Ayache2015}
A.~Ayache and J.~Hamonier.
\newblock Linear multifractional stable motion: wavelet estimation of {H}(.)
  and $\alpha$ parameter.
\newblock {\em Lithuanian {M}athematical {J}ournal}, 55(2):159--192, 2015.

\bibitem{Ayache2017}
A.~Ayache and J.~Hamonier.
\newblock Linear {M}ultifractional {S}table {M}otion: an a.s. uniformly
  convergent estimator for {H}urst function.
\newblock {\em Bernoulli}, 23(2):1365--1407, 2017.

\bibitem{Ayache2007}
A.~Ayache, S.~Jaffard, and M.S. Taqqu.
\newblock Wavelet construction of generalized multifractional processes.
\newblock {\em Revista {M}atem\'atica {I}beroamericana}, 23(1):327--370, 2007.

\bibitem{Ayache2004}
A.~Ayache and J.~L\'evy V\'ehel.
\newblock On the identification of the pointwise {H}\"older exponent of the
  generalized multifractional {B}rownian motion.
\newblock {\em Stochastic Processes and their Applications}, 111(1):119--156,
  2004.

\bibitem{Bardet2013}
J.M. Bardet and D.~Surgailis.
\newblock Nonparametric estimation of the local {H}urst function of
  multifractional processes.
\newblock {\em Stochastic {P}rocesses and their {A}pplications},
  123(3):1004--1045, 2013.

\bibitem{Benassi2000}
A.~Benassi, P.~Bertrand, S.~Cohen, and J.~Istas.
\newblock Identification of the {H}urst index of a step fractional {B}rownian
  motion.
\newblock {\em Statistical Inference for Stochastic Processes}, 3(1):101--111,
  2000.

\bibitem{Benassi1998}
A.~Benassi, S.~Cohen, and J.~Istas.
\newblock Identifying the multifractional function of a {G}aussian process.
\newblock {\em Statistics and Probability Letters}, 39(4):337--345, 1998.

\bibitem{Benassi1997}
A.~Benassi, S.~Jaffard, and D.~Roux.
\newblock Gaussian processes and pseudodifferential elliptic operators.
\newblock {\em Revista {M}athematica {I}beroamericana}, 13(1):19--90, 1997.

\bibitem{Coeurjolly2005}
J.F. Coeurjolly.
\newblock Identification of multifractional {B}rownian motion.
\newblock {\em Bernoulli}, 11(6):987--1008, 2005.

\bibitem{Coeurjolly2006}
J.F. Coeurjolly.
\newblock Erratum: {I}dentification of multifractional {B}rownian motion.
\newblock {\em Bernoulli}, 12(2):381--382, 2006.

\bibitem{Cohen2013}
S.~Cohen and J.~Istas.
\newblock {\em Fractional fields and applications}.
\newblock Springer-{V}erlag, Berlin, 2013.

\bibitem{DanIst2017}
T.T.N. Dang and J.~Istas.
\newblock Estimation of the {H}urst and the stability indices of a
  {H}-self-similar stable process.
\newblock {\em Electronic Journal of Statistics}, 11(2):4103--4150, 2017.

\bibitem{Embrechts2002}
P.~Embrechts and M.~Maejima.
\newblock {\em Self-similar processes}.
\newblock Princeton {U}niversity {P}ress, Princeton, {N}{J}, 2002.

\bibitem{Falconer2009}
K.J. Falconer, R.~Le~Gu\'evel, and J.~L\'evy~V\'ehel.
\newblock Localizable moving average symmetric stable and multistable
  processes.
\newblock {\em Stochastic {M}odels}, 25(4):648--672, 2009.

\bibitem{Vehel2009}
K.J. Falconer and J.~L\'evy~V\'ehel.
\newblock Multifractional, multistable, and other processes with prescribed
  local form.
\newblock {\em Journal of {T}heoretical {P}robability}, 22(2):375--401, 2009.

\bibitem{Jacques2012}
J.~Istas.
\newblock Estimating self-similarity through complex variations.
\newblock {\em Electronic {J}ournal of {S}tatistics}, 6:1392--1408, 2012.

\bibitem{Lacaux2004}
C.~Lacaux.
\newblock Real harmonizable multifractional {L}\'evy motions.
\newblock {\em Annal de l'{I}nstitut {H}enry {P}oincar\'e ({B}) {P}robability
  and {S}tatistics}, 40(3):259--277, 2004.

\bibitem{Guevel2013}
R.~Le~Gu\'evel.
\newblock An estimation of the stability and the localisability functions of
  multistable processes.
\newblock {\em Electronic {J}ournal of {S}tatistics}, 7:1129--1166, 2013.

\bibitem{Yimin2008}
M.~Meerschaert, D.~Wu, and Y.~Xiao.
\newblock Local times of multifractional {B}rownian sheets.
\newblock {\em Bernoulli}, 14(3):865--898, 2008.

\bibitem{Peltier1995}
R.F. Peltier and J.~L\'evy~V\'ehel.
\newblock Multifractional {B}rownian motion: definition and preliminary
  results.
\newblock Technical Report 2645, Rapport de recherche de l'INRIA, 1995.

\bibitem{Taqqu1994}
G.~Samorodnitsky and M.S. Taqqu.
\newblock {\em Stable {N}on-{G}aussian {R}andom {P}rocesses}.
\newblock Chapmann and {H}all, New {Y}ork, 1994.

\bibitem{Stoev2002}
S.~Stoev, V.~Pipiras, and M.S. Taqqu.
\newblock Estimation of the self-similarity parameter in linear fractional
  stable motion.
\newblock {\em Signal {P}rocessing}, 82:1873--1901, 2002.

\bibitem{Stoev2004}
S.~Stoev and M.S. Taqqu.
\newblock Stochastic properties of the linear multifractional stable motion.
\newblock {\em Advances in {A}pplied {P}robability}, 36(4):1085--1115, 2004.

\bibitem{Stoev2006}
S.~Stoev and M.S. Taqqu.
\newblock How rich is the class of multifractional {B}rownian motions?
\newblock {\em Stochastic {P}rocesses and their {A}pplications},
  116(2):200--221, 2006.

\end{thebibliography}
\end{document}